\documentclass[final,onefignum,leqno,onetabnum]{siamltex}
\usepackage{amsmath,epsfig,cmmib57}
\usepackage{amssymb}
\usepackage{latexsym}
\usepackage{times}
\usepackage{multirow}
\usepackage{graphicx}
\usepackage{subfigure}	% add 1.22
\usepackage{placeins}	% add 1.22
\usepackage{epstopdf}
\begin{document}
\title{A two-level overlapping Schwarz method with energy-minimizing multiscale
coarse basis functions
\thanks{The research of Junxian Wang is supported by National Natural Science Foundation of China (Project number 11201398, 11301448). The research of Eric Chung is supported by Hong Kong RGC General Research Fund (Project number 14304217)
and CUHK Faculty of Science Research Incentive Fund 2017-18. The research of Hyea Hyun Kim is
supported by the National Research Foundation of Korea(NRF) grants
funded by NRF-2015R1A5A1009350.} }

\author{
Junxian Wang \thanks{School of Mathematics and Computational Science, Xiangtan University, Xiangtan, Hunan 411105, China.}
\and
Eric Chung \thanks{Department of Mathematics, The Chinese University of Hong Kong, Hong Kong SAR.}
\and
Hyea Hyun Kim \thanks{Department of Applied Mathematics and Institute of Natural Sciences, Kyung Hee University, Korea.
Electronic mail address: hhkim@khu.ac.kr.}
}
%This work was supported in part by the National
%Science Foundation under Grants NSF-CCR-9503408 and NSF-ODURF-354151,
%and in part by the U.S. Department of Energy under
%contract DE-FG02-92ER25127.

\maketitle

%\begin{center}
%\today
%\end{center}

\begin{abstract}
A two-level overlapping Schwarz method is developed for second order elliptic problems
with highly oscillatory and high contrast coefficients, for which
it is known that the standard coarse problem fails to give a robust preconditioner.
In this paper, we develop energy minimizing multiscale finite element functions to form a more robust coarse problem.
First, a local spectral problem is solved in each non-overlapping coarse subdomain, and dominant eigenfunctions
are selected as auxiliary functions, which are crucial for the high contrast case.
The required multiscale basis functions are then obtained by minimizing an energy
subject to some orthogonality conditions with respect to the auxiliary functions.
Due to an exponential decay property, the minimization problem is solved
locally on oversampling subdomains, that are unions of a few coarse subdomains.
The coarse basis functions are therefore local and can be computed efficiently.
The resulting preconditioner is shown to be robust with respect to the contrast in the coefficients
as well as the overlapping width in the subdomain partition.
Numerical results are presented to validate the theory and show the performance.
%A two-level overlapping Schwarz method is developed for second order elliptic problems
%with highly oscillatory and high contrast coefficients.
%The standard coarse problem often fails to give a robust preconditioner.
%Multiscale finite element functions are utilized to form a more robust coarse problem.
%The multiscale finite element functions have some nice orthogonality to an auxiliary
%function space, which is formed by taking problematic components of the solution
%related to high and random variations in the model coefficients.
%Such orthogonality results in a preconditioner robust to the contrast in the coefficients
%as well as the overlapping width in the subdomain partition. A more practical version of
%the proposed method is also analyzed. Numerical results are included for the practical
%version.
\end{abstract}

\begin{keywords}
overlapping Schwarz method, high contrast, multiscale finite element basis, coarse problem
\end{keywords}

{ \bf AMS(MOS) subject classifications.} 65F10, 65N30, 65N55
\pagestyle{myheadings} \thispagestyle{plain} \markboth{Junxian Wang, Eric Chung, and Hyea Hyun Kim}{A TWO-LEVEL OVERLAPPING SCHWARZ ALGORITHM}

\section{Introduction}
A two-level overlapping Schwarz method is proposed for solving the algebraic equation obtained from the
finite element discretization of the following second order elliptic problem
\begin{equation}\label{model:equation}
\left\{
\begin{array}{rcll}
    -\nabla \cdot (\rho(x) \nabla u(x) ) &=& f(x), & \text{in}~\Omega,\\
    u&=&0, & \text{on}~\partial\Omega,
\end{array}
\right.
\end{equation}
where $\Omega$ is a polygonal (polyhedral) domain in $\mathbb{R}^r$  and $f(x)$ is in $L^2(\Omega)$,
the space of square integrable functions.
The coefficient $\rho(x)$ in the above model problem is highly heterogenous with high contrast
inside the domain $\Omega$. For an accurate approximation, very fine meshes are required to resolve the variations
in the coefficient and thus the resulting algebraic system can become very large.
In addition, the condition number of the algebraic system highly depends on the contrast in the coefficients.
For fast solutions of the algebraic system, we will propose a two-level overlapping Schwarz preconditioner.
We note that the standard coarse problem based on a coarse mesh
often fails to give a robust preconditioner~\cite{Galvis1}.
In this paper, the new idea is that the multiscale finite element functions proposed in \cite{CEL2018} are utilized to form a more robust coarse problem.

In \cite{CEL2018}, constrained energy minimizing multiscale
finite element functions are introduced for approximating the solution of a multiscale model problem
and the approximate solutions are shown to converge with the errors linearly decreasing with respect to the coarse mesh size and independent of the contrast in the coefficient $\rho(x)$.
By using the constrained energy minimizing basis functions, we will form the coarse component of
the two-level overlapping Schwarz method.
The proposed coarse space will give the
preconditioner of which performance is robust to both the high variation in the coefficient $\rho(x)$
and the overlapping width in the subdomain partition. This is improvement over
the previous study~\cite{Galvis1,Efendiev-EM2NA}, where the condition number depends on the overlapping width
in the partition.
We remark that a similar coarse space is used in \cite{hou2017sparse} for higher order elliptic operators.

%As we will see later, to find
Our theory begins with the global constrained energy minimizing functions. To find these functions, we need to solve the model problem
in the whole finite element space $V_h(\Omega)$.  These global functions are able to produce very robust preconditioners but
the resulting method is not efficient.
We thus propose a more practical method, where we solve the same problem in a smaller finite element space
$V_h(\widetilde{\Omega}_i)$, the restriction of $V_h(\Omega)$ to the space $H_0^1(\widetilde{\Omega}_i)$
and $\widetilde{\Omega}_i$ is a subregion containing a subdomain $\Omega_i$.
This approach is similar to the oversampling idea in multiscale finite element methods,
and is based on an exponential decay property of the global constrained energy minimizing functions.
Using these more practical functions, we can form a coarse problem for the two-level overlapping Schwarz
preconditioner. In addition, we can provide a complete analysis for the estimate of condition numbers for the preconditioner.
In detail, when the size of the oversampling region $\widetilde{\Omega}_i$
is large enough then the condition numbers are shown to be robust to the contrast in the coefficient
as well as the overlapping width in the partition. In numerical results, we can observe
quite robust results even for a small oversampling region, where
the region is formed by including only one layer of neighboring coarse meshes from $\Omega_i$.

We note that similar approaches are considered in other types of domain decomposition preconditioners.
In those works, the coarse problem is formed by enriching the standard coarse space
with eigenvectors adaptively chosen from appropriate generalized eigenvalue problems
on each subdomain or on each subdomain interface. We refer \cite{Mandel-CMAME-2007,Mandel-ABDDC,Dolean3,S-R-ref-2013,Klawonn:PAMM:2014,KimChung:2014,Klawonn-3d-2016,CW-3d-2016,PD-ETNA-2017,
IGA-adaptive-BDDC-2017,KimChungWang:2015}
for the works under the BDD(C) and FETI-DP framework and
\cite{Galvis1,Galvis2,Efendiev-EM2NA,Dolean1,Dolean2} for the works under the two-level additive Schwarz framework.
Our work is similar to that considered in \cite{Galvis1} and the main contribution of our work is the construction of a more robust coarse problem.

This paper is organized as follows. In Section~\ref{sec:coarsebasis}, the constrained energy minimizing functions
introduced in \cite{CEL2018} are defined using two bilinear forms that are relevant to the two-level
overlapping Schwarz framework. In Sections~\ref{sec:oS} and \ref{sec:AN}, the two-level overlapping
Schwarz preconditioner equipped with the coarse problem from the constrained energy minimizing functions
is proposed and its condition number bound is analyzed. In Section~\ref{sec:practical}, more practical coarse
basis functions are proposed and utilized to form the coarse problem. In addition,
an extensive analysis for the corresponding preconditioner
is carried out.
In Section~\ref{sec:numerics}, numerical results are presented to confirm the theoretical estimate
for the practical coarse problem.

\section{Constrained energy minimizing multiscale basis functions}\label{sec:coarsebasis}
We equip a conforming triangulation $\mathcal{T}_h$ for $\Omega$ and introduce $V_h$ as the standard conforming finite element space of piecewise linear functions corresponding to $\mathcal{T}_h$ with the zero value on $\partial \Omega$.
The Galerkin approximation to the model problem in \eqref{model:equation} then gives that:
find $u_h$ in $V_h$ such that
\begin{equation}\label{Galerkin:modelpb}
a(u_h,v)=(f,v),\quad \forall v \in V_h,
\end{equation}
where
$$a(u_h,v):=\int_{\Omega} \rho(x) \nabla u_h \cdot \nabla v \, dx,\quad
(f,v):=\int_{\Omega} fv \, dx.$$
We assume that the triangulation $\mathcal{T}_h$ is fine enough to resolve the variation in the coefficient $\rho(x)$, i.e.,
for a given constant $C$, the triangulation $\mathcal{T}_h$ satisfies that
\begin{equation}\label{assume:Th}
\frac{\max_{x \in \tau} \rho(x) }{\min_{x \in \tau} \rho(x) } \le C,\quad \forall \tau \in \mathcal{T}_h.
\end{equation}
The Galerkin approximation results in the following algebraic system
\begin{equation}\label{pb:algeqn} A U = F,
\end{equation}
for which we will propose a two-level overlapping Schwarz preconditioner robust
to the variations and contrast in the coefficient $\rho(x)$ and to the overlapping width in the subdomain partition.
For that purpose, we first form an auxiliary space by solving a certain generalized eigenvalue
problem in each subdomain and then find energy minimizing coarse basis functions with certain orthogonality
conditions with respect to the auxiliary space.

We partition the domain $\Omega$ into non-overlapping subdomains $\{\Omega_i\}_{i=1}^N$
where each $\Omega_i$ is a connected union of triangles in $\mathcal{T}_h$.
We then extend each subdomain by several layers of triangles in $\mathcal{T}_h$ to obtain an overlapping
subdomain partition $\{ \Omega_i^\prime \}_{i=1}^N$.
We use the notation $2 \delta$ for the minimum overlapping width in the overlapping subdomain partition, $H$ for the
maximum subdomain diameter in the non-overlapping subdomain partition, and $h$ for the maximum triangle diameter in the triangulation $\mathcal{T}_h$.
For the given overlapping subdomain partition, we introduce a partition of unity $\{ \theta_i(x) \}_{i=1}^N$,
where $\sum_{i=1}^N \theta_i(x)=1$ and each $\theta_i(x)$ is supported in $\Omega_{i}^\prime$.
%We note that our coarse component will be formed based on the non-overlapping subdomain partition
%while the functions $\theta_i(x)$ are defined on the overlapping subdomain partition.
In our work, we may assume that the partition of unity functions are in the space $V_h$.
We can choose a function $\theta_i(x)$ in $V_h$ with the following nodal values at any node $x$
interior to $\Omega_i^\prime$,
$$\theta_i(x)=\frac{1}{\mathcal{N}(x)}$$
and zero value at the rest nodes.
In the above, $\mathcal{N}(x)$ denotes the number of overlapping subdomains containing the node $x$.

We consider the following generalized eigenvalue problem in each non-overlapping subdomain $\Omega_i$:
\begin{equation}\label{pb:Geig}
a_i(\phi^{(i)}_j, w)=\lambda_j^{(i)} s_i( \phi^{(i)}_j, w),\quad \forall w \in V(\Omega_i),
\end{equation}
where $V(\Omega_i)$ is the restriction of functions in $V_h$ to the subdomain $\Omega_i$
and the above bilinear forms are defined as
$$a_i (v,w):=\int_{\Omega_i} \rho(x) \nabla v \cdot \nabla w\, dx,\quad
s_i (v,w):=\int_{\Omega_i} \rho(x) \sum_{l \in n(i)} |\nabla \theta_l(x)|^2 v\, w\, dx.$$
In the above, $n(i)$ denotes the set of overlapping subdomain indices $l$, i.e., $\Omega_l^\prime$, such that the support of the corresponding partition of unity function $\theta_l(x)$
has a nonempty intersection with $\Omega_i$.
We assume that the eigenvalues $\lambda^{(i)}_j$ are arranged in ascending order and we choose
the eigenvectors $\phi_j^{(i)}$ with their associate eigenvalues $\lambda_j^{(i)}$ smaller than
a given tolerance value $\Lambda$, i.e., $\lambda_j^{(i)} < \Lambda$.
We use the notation $l_i$ for the number of such eigenvectors.

We now obtain a set of auxiliary multiscale finite element functions by collecting
all the selected eigenvectors
$$V_{aux}:=\left\{  \phi_{j}^{(i)}\,| \, i=1,\cdots,N,\; j=1,\cdots,l_i\right\}.$$
We assume that $\phi_j^{(i)}$ are normalized, i.e., $s_i(\phi_j^{(i)},\phi_j^{(i)})=1$, $j=1,\cdots,l_i$.
We introduce the following definition for a function $v$ in $V_h$:
$v$ is $\phi_j^{(i)}$-orthogonal if $s_i(v,\phi_j^{(i)})=1$
and $s_k(v,\phi_l^{(k)})=0$ for $k \neq i,\; l=1,\cdots,l_k$ or $k=i,\; l=1,\cdots,j-1,j+1,\cdots,l_i$.
We obtain a set of coarse basis functions $\psi_{j}^{(i)}$ from the solution
of the following constrained minimization problem:
\begin{equation}\label{pb:cemin}
\psi_j^{(i)}=\text{argmin} \{ a(\psi,\psi) \, | \, \psi \in V_h,\; \psi \text{ is } \phi_j^{(i)}\text{-orthogonal}. \}
\end{equation}

We note that we can solve the above constrained minimization problem by introducing Lagrange multipliers $\eta_l^{(k)}$
for the constraints and form the following mixed problem: find $\psi_j^{(i)}$ and $\eta_l^{(k)}$ such that
\begin{eqnarray}\label{pb:mixed}
a(\psi_j^{(i)},v)+ \sum_{k=1}^N \sum_{l=1}^{l_k} \eta_l^{(k)} s_k(v,\phi_l^{(k)}) &=&0,\quad \forall v \in V_h \\
s_k(\psi_j^{(i)},\phi_l^{(k)})&=&\delta_{k,l}^{i,j},\; k=1,\cdots,N,\; l=1,\cdots,l_k,
\end{eqnarray}
where $\delta^{i,j}_{k,l}$ is one when $k=i$ and $l=j$, and its value is zero, otherwise.
We now define the space of coarse basis functions
$$V_{glb}=\text{span} \{ \psi_{j}^{(i)},\; i=1,\cdots,N,\; j=1,\cdots,l_i\}.$$

We introduce
\begin{equation}\label{wtV}
\widetilde{V}:=\{ v \in V_h \,|\, s_k(v, \phi_{l}^{(k)})=0,\, \forall k=1,\cdots,N, l=1,\cdots,l_k\}.
\end{equation}
By the first equation in \eqref{pb:mixed}, we observe the following orthogonal property
$$a(\psi_{j}^{(i)}, v)=0,\quad \forall v \in \widetilde{V}$$
and thus obtain that
$$V_h=\widetilde{V} \oplus V_{glb}.$$
We note that $V_h={V}_{glb}^\perp \oplus V_{glb}$ and $\widetilde{V}$ is contained in $V_{glb}^\perp$.
Since the dimension of $V_{glb}^{\perp}$ is equal to the dimension of $\widetilde{V}$, we have
$\widetilde{V}=V_{glb}^{\perp}$.
For a proof, see \cite{CEL2018}.

As proposed in \cite{CEL2018}, we can consider a more practical relaxed constrained energy minimizing problem:
\begin{equation}\label{relaxed-energy-min}\psi_j^{(i)}=\text{argmin}\left\{ a(\psi,\psi)+s_i(\pi_i \psi - \phi_j^{(i)}, \pi_i \psi - \phi_j^{(i)})+\sum_{k \ne i} s_k(\pi_k \psi, \pi_k \psi) \,|\,
\forall \psi \in V_h \right\},
\end{equation}
where
\begin{equation}\label{def:pi}\pi_k \psi:=\sum_{j=1}^{l_k} s_k(\psi,\phi_j^{(k)}) \phi_j^{(k)}.
\end{equation}
We note that the function $\psi_j^{(i)}$ in \eqref{relaxed-energy-min} can be found by solving the following problem:
find $\psi_j^{(i)}$ in $V_h$ such that
\begin{equation}\label{eqn:relaxed-energy-min}
a(\psi_j^{(i)}, v)+ \sum_{k=1}^N s_k(\pi_k \psi_j^{(i)}, \pi_k v)=s_i(\phi_j^{(i)}, \pi_i v ),\quad \forall v \in V_h.
\end{equation}
Let $V_{glb}$ be obtained from $\psi_j^{(i)}$ of the above relaxed constrained problem.
We can then observe the same property for $V_{glb}$ as before (\cite{CEL2018}), i.e.,
$$V_h=\widetilde{V} \oplus V_{glb}.$$
In the following, we will use the space $V_{glb}$ from the relaxed constrained problem~\eqref{relaxed-energy-min} as the coarse space of
the two-level overlapping Schwarz algorithm.

\section{Two-level overlapping Schwarz algorithm}\label{sec:oS}
In this section, we propose a two-level overlapping Schwarz preconditioner for the algebraic equation
in \eqref{pb:algeqn}.
We note that we will use the functions in $V_{glb}$ to form the coarse problem of the preconditioner
and the overlapping subdomain partition $\{ \Omega_i^\prime \}_{i=1}^N$ to form
the local problems of the preconditioner.

We introduce the local finite element space $V_0(\Omega_i^\prime)$, which is the restriction
of functions in $V_h$ to $\Omega_i^\prime$ with the zero value on $\partial \Omega_i^\prime$.
We define the local problem matrix by
$$\langle A_{i} v, w \rangle :=\int_{\Omega_i^\prime} \rho(x) \nabla v \cdot \nabla w \, dx,\; \forall
v,w \in V_0(\Omega_i^\prime).$$
We introduce the restriction $R_i$ from $V_h$ to $V_0(\Omega_i^\prime)$
and denote by $R_i^T$ the extension from $V_0(\Omega_i^\prime)$ to $V_h$ by zero.

We define the coarse problem
matrix by
$$A_0=a(\psi_j^{(i)}, \psi_r^{(k)}), \forall \psi_j^{(i)}, \psi_r^{(k)} \in V_{glb}.$$
We note that the size of the matrix $A_0$ is identical to the dimension of $V_{glb}$.
We introduce $R_0$ by the matrix with its rows consisting of nodal values of $\psi_j^{(i)}$ in $V_{glb}$
and define the two-level overlapping Schwarz preconditioner as
\begin{equation}\label{osprecond}
\sum_{i=1}^N R_i^T A_i^{-1} R_i + R_0^T A_0^{-1} R_0.
\end{equation}

\section{Analysis of condition numbers}\label{sec:AN}
For the overlapping Schwarz method, the upper bound estimate can be obtained from the coloring argument.
We will only need to work on the following lower bound estimate, see \cite{TW-Book} for the abstract theory of the two-level overlapping Schwarz method:
\begin{lemma}\label{lemma:stable:decomposition:Vglb}
For any given $u$ in $V_h$, there exists $\{u_i \}_{i=0}^N$ with $u_0 \in V_{glb}$, and $u_i \in V_0(\Omega_i^\prime)$, $i\ge1$, such that
$$u=\sum_{i=1}^N u_i + u_0$$
and
$$\sum_{i=1}^N a(u_i,u_i) + a(u_0,u_0) \le C_0^2 a(u,u)$$
with the constant $C_0$ dependent on $\Lambda$ but independent of $\rho(x)$ and $\delta$.
\end{lemma}
\begin{proof}
We will choose $u_0$ as the solution of
$$a(u_0,v)=a(u,v),\quad \forall v \in V_{glb}$$
and choose $u_i$ as
$$u_i=I^h(\theta_i(u-u_0)),$$
where $I^h(v)$ denotes the nodal interpolant of $v$ to the space $V_h$.
We note that $u-u_0$ is in $V_{glb}^\perp$ and also in $\widetilde{V}$, since $V_{glb}^\perp=\widetilde{V}$.
This nice property of $u-u_0$ will be used in the following estimates.

We can see that $u_i$ is supported in $\Omega_{i}^\prime$ by the construction and
then obtain that
\begin{align}
&\sum_{i=1}^N a(u_i,u_i) + a(u_0,u_0) \nonumber \\
&= \sum_{i=1}^N \int_{\Omega_{i}^\prime} \rho | \nabla I^h( \theta_i (u-u_0))|^2 \, dx + a(u_0,u_0) \nonumber \\
&\le C_I\sum_{i=1}^N \int_{\Omega_{i}^\prime} \rho |\nabla( \theta_i (u-u_0))|^2 \, dx + a(u_0,u_0)\label{star1} \\
&\le 2 C_I \sum_{i=1}^N \left( \int_{\Omega_{i}^\prime} \rho |\nabla(u-u_0)|^2\, dx+
\int_{\Omega_{i}^\prime} \rho |\nabla \theta_i |^2 (u-u_0)^2\, dx \right) + a(u_0,u_0)\nonumber\\
&\le 2 C C_I \sum_{i=1}^N \int_{\Omega_{i}} \rho |\nabla(u-u_0)|^2\, dx+
2 C_I \sum_{i=1}^N \int_{\Omega_{i}} \rho \sum_{k \in n(i)} |\nabla \theta_k |^2 (u-u_0)^2\, dx + a(u_0,u_0) \nonumber \\
&= 2 C C_I \sum_{i=1}^N \int_{\Omega_{i}} \rho |\nabla(u-u_0)|^2\, dx+
2 C_I \sum_{i=1}^N s_i(u-u_0, u-u_0) + a(u_0,u_0) \nonumber \\
&\le 2 C C_I \sum_{i=1}^N (1+ \Lambda^{-1}) \int_{\Omega_{i}} \rho |\nabla(u-u_0)|^2 \, dx + a(u_0,u_0) \label{star2} \\
&= 2 C C_I (1+\Lambda^{-1} ) a(u-u_0,u-u_0)+ a(u_0,u_0) \nonumber \\
&\le 2 C C_I (1+\Lambda^{-1} ) \left( a(u-u_0,u-u_0)+a(u_0,u_0) \right) \nonumber \\
&= 2 C C_I (1+\Lambda^{-1}) a(u,u), \label{star3}
\end{align}
where the constant $C_I$ depends on the stability of the nodal interpolation $I^h$, the constant $C$ depends on the maximum number of overlapping subdomains sharing the same location in $\Omega$, the notation $n(i)$ means the set of
overlapping subdomain indices $l$
such that $\Omega_l^{\prime}$ intersects with $\Omega_i$.
In \eqref{star1}, we use the assumption~\eqref{assume:Th} on the triangulation $\mathcal{T}_h$.
We note that $u-u_0$ is in $V_{glb}^\perp(=\widetilde{V})$ and thus in \eqref{star2} we can use the following inequality with the constant $\Lambda^{-1}$
$$s_i(u-u_0,u-u_0) \le \Lambda^{-1} a_i(u-u_0,u-u_0)$$
and finally obtain \eqref{star3}. We note that we obtain the constant $C_0^2=2 C C_I (1+ \Lambda^{-1})$ independent
of $\rho(x)$ as well as the overlapping width, which is improvement over the previous works~\cite{Galvis1,Galvis2,Efendiev-EM2NA}.
\end{proof}

We note that the computation of $\psi_j^{(i)}$ requires solution of the relaxed constrained minimization problem in
the global finite element space $V_h$. In practice, we can solve the same problem in a subspace of $V_h$,
where the functions are restricted to the local region $\widetilde{\Omega}_{i}$ containing $\Omega_i$.
In a more detail, we solve
$$\psi_{j,ms}^{(i)}=\text{argmin}\left\{  a(\psi,\psi)+s_i(\pi_i(\psi)-\phi_j^{(i)},\pi_i(\psi)-\phi_j^{(i)})
+\sum_{k \ne i} s_k (\pi_k \psi, \pi_k \psi) \, \left|
\, \forall \psi \in \widetilde{V}_i \right. \right\},$$
where $\widetilde{V}_i$ denotes the restriction of $V_h$ to the subregion $\widetilde{\Omega}_i$
with zero value on $\partial \widetilde{\Omega}_i$, i.e., $\widetilde{V}_i=V_h \bigcap H_0^1(\widetilde{\Omega}_i)$.
From the above minimization problem,
we obtain $\psi_{j,ms}^{(i)}$ and denote by $\Psi_{j,ms}^{(i)}$ the extension of $\psi_{j,ms}^{(i)}$
by zero to the function in $V_h$.
We then define $V_{ms}$ by
$$V_{ms}:=\text{span}\{ \Psi_{ms,j}^{(i)}\,|\, i=1,\cdots,N,\;j=1,\cdots,l_i \}.$$
We can propose the following more practical preconditioner
\begin{equation}\label{precond:Mms}M^{-1}_{ms}=\sum_{i=1}^N R_i^T A_i^{-1} R_i + R_{0,ms}^T A_{0,ms}^{-1} R_{0,ms},
\end{equation}
where $A_{0,ms}$ and $R_{0,ms}$ are defined similarly as before by
replacing $V_{glb}$ with $V_{ms}$.
We remark that the choice of $\widetilde{\Omega}_{i}$ will be discussed next.
In short, it is an oversampled region of $\Omega_i$ obtained by extending it by several neighboring
subdomains.
Hence, computing (\ref{precond:Mms}) is relatively cheap.

\section{Analysis of condition numbers using $V_{ms}$ as a coarse space}\label{sec:practical}
In this section, we will provide a rigorous proof for analysis of condition numbers
for the two-level overlapping Schwarz preconditioner $M^{-1}_{ms}$, where a more practical coarse space
$V_{ms}$ is employed.

We recall the function $u_0$ in $V_{glb}$ obtained from
$$a(u_0, v)=a(u,v),\quad \forall v \in V_{glb}$$
and express $u_0$ as
$$u_0=\sum_{i=1}^N \sum_{j=1}^{l_i} c_{ij} \psi_j^{(i)}.$$
We then introduce
\begin{equation}\label{def:ums}u_{ms}=\sum_{i=1}^N \sum_{j=1}^{l_i} c_{ij} \Psi_{j,ms}^{(i)}
\end{equation}
with the same coefficients $c_{ij}$ as in $u_0$,
and
\begin{equation}\label{def:ui}
u_i=I^h(\theta_i (u-u_{ms})).
\end{equation}
We can then decompose $u$ as a sum of these functions,
$$u= \sum_{i=1}^N u_i + u_{ms}$$
and we will prove that
$$\sum_{i=1}^N a(u_i,u_i)+ a(u_{ms},u_{ms}) \le C a(u,u),$$
where the constant $C$
depends on $\Lambda$ but does not depend on $\rho(x)$ and the overlapping width in the partition
$\{ \Omega_i^\prime \}$.

We first obtain that
\begin{align}
a(u_{ms},u_{ms})&\le  2 a(u,u) + 2 a(u-u_{ms}, u-u_{ms}) \nonumber \\
 &\le  2 a(u,u)+4 a(u-u_0,u-u_0)+ 4 a(u_0-u_{ms},u_0-u_{ms}) \nonumber \\
 &\le 6 a(u,u) + 4 a(u_0-u_{ms},u_0-u_{ms}), \label{coarse:term}
\end{align}
where $a(u-u_0,u-u_0)$ can be bounded by $a(u,u)$ using the orthogonality.
In the following, we will use the notation $\sum_i$ for $\sum_{i=1}^N$ for brevity.

We now consider
\begin{eqnarray*}
&& \sum_i a(u_i,u_i)=\sum_i a(I^h(\theta_i (u-u_{ms})), I^h(\theta_i (u-u_{ms})) ) \nonumber \\
& \le & C_I \sum_i a(\theta_i (u-u_{ms}), \theta_i (u-u_{ms})) \nonumber \\
& \le & 2 C_I \sum_i \int_{\Omega_i^\prime} \rho(x) |\nabla (u-u_{ms})|^2 \, dx
+ 2 C_I \sum_i \int_{\Omega_i^\prime} \rho(x) |\nabla \theta_i(x)|^2 (u-u_{ms})^2 \, dx \nonumber \\
& \le & 2 C C_I \sum_i \int_{\Omega_i} \rho(x) |\nabla(u-u_{ms})|^2 \, dx
+ 2 C_I \sum_i \int_{\Omega_i} \rho(x) \sum_{l\in n(i)} |\nabla \theta_l(x)|^2 (u-u_{ms})^2 \, dx \nonumber \\
&\le & 2 C C_I \left( a(u-u_{ms},u-u_{ms}) + \sum_{i} s_i(u-u_{ms},u-u_{ms}) \right).\label{bd:pb2}
\end{eqnarray*}
Using that $$a(u-u_0,u-u_0)+\sum_i s_i(u-u_0,u-u_0) \le (1+\Lambda^{-1}) a(u,u),$$
we obtain
\begin{align*}
\sum_i a(u_i,u_i) & \le 4 C C_I \left((1+\Lambda^{-1})a(u,u) + a(u_0-u_{ms},u_0-u_{ms}) \right. \\
&  \left. +\sum_i s_i(u_0-u_{ms},u_0-u_{ms})\right).
\end{align*}
Combining \eqref{coarse:term} and the above,
the estimate for $\sum_i a(u_i,u_i)+ a(u_{ms},u_{ms}) < C a(u,u)$ is reduced to
\begin{equation}\label{estimate:main}
a(u_0-u_{ms},u_0-u_{ms})+\sum_{i} s_i(u_0-u_{ms},u_0-u_{ms})
\le C_0^2 a(u,u).
\end{equation}

For that purpose, we will obtain several preliminary inequalities below.
We recall the definition for $\widetilde{V}_i$, i.e.,
$\widetilde{V}_i:=V_h \bigcap H_0^1(\widetilde{\Omega}_i)$, where $\widetilde{\Omega}_i$
is the local subregion containing $\Omega_i$, which will be defined later.
\begin{lemma}\label{lemma:min:property}
For any $v$ in $\widetilde{V}_i$, we obtain that
\begin{align*}
&&a(\psi_j^{(i)}-\Psi_{j,ms}^{(i)},\psi_j^{(i)}-\Psi_{j,ms}^{(i)})
+\sum_l s_l( \pi_l (\psi_j^{(i)}-\Psi_{j,ms}^{(i)}),\pi_l (\psi_j^{(i)}-\Psi_{j,ms}^{(i)}))\\
&&\le
a(\psi_j^{(i)}-v,\psi_j^{(i)}-v)
+\sum_l s_l( \pi_l (\psi_j^{(i)}-v),\pi_l (\psi_j^{(i)}-v)).
\end{align*}
\end{lemma}
\begin{proof}
Using the orthogonality,
$$a(\psi_j^{(i)}-\Psi_{j,ms}^{(i)},v)+\sum_l s_l(\pi_l( \psi_j^{(i)}-\Psi_{j,ms}^{(i)}), \pi_l v)=0,\quad
\forall v \in \widetilde{V}_i,$$
we can obtain the above inequality.
\end{proof}

For simplicity, we will use the following notations below,
$$|v|_{a(R)}^2=\int_{R} \rho(x) | \nabla v|^2 \, dx,\quad \| v \|_{s_l}^2:=s_l (v , v).$$
We introduce a region $\Omega_{i,k}$ by extending $\Omega_i$ by $k$ layers of neighboring subdomains $\Omega_l$
and define a function
$$\chi_i^k=\sum_{\Omega_l \subset \Omega_{i,k}} \theta_{l},$$
where $\theta_l$ is the partition of unity function defined for the overlapping subdomains $\Omega_l^\prime$,
that are obtained by extending $\Omega_l$ with several layers of fine meshes (i.e. elements in $\mathcal{T}_h$).
We let $d$ be the number of such fine mesh layers and then the overlapping width $2 \delta$  becomes
$2dh$ in the resulting overlapping partition $\{\Omega_l^\prime\}$.
We then have the following property of $\chi_i^k$,
$$\chi_i^k=1 \text{ on } \Omega_{i,k}^{\delta-}, \quad \chi_i^k=0 \text{ on } \Omega \setminus \Omega_{i,k}^{\delta},$$
where $\Omega_{i,k}^{\delta-}$ is the subset of $\Omega_{i,k}$ which is obtained by deleting $d$ layers of fine  meshes from $\Omega_{i,k}$ and $\Omega_{i,k}^\delta$ is obtained by extending $\Omega_{i,k}$ by $d$ layers
of fine meshes. In the above Lemma~\ref{lemma:min:property}, one can choose $\widetilde{\Omega}_i$ as $\Omega_{i,k}^\delta$
and thus choose $v=I^h( \chi_i^k \psi_j^{(i)})\in \widetilde{V}_i$ to obtain the following estimate:
\begin{lemma}\label{lemma:estimate:difference}
We obtain that
\begin{align*}
&| \psi_j^{(i)}-\Psi_{j,ms}^{(i)}|_{a(\Omega)}^2 + \sum_l \| \pi_l (\psi_j^{(i)}-\Psi_{j,ms}^{(i)}) \|_{s_l}^2 \\
\le & C_I \left(  (1+\Lambda^{-1}) | \psi_j^{(i)}|_{a(\Omega \setminus \Omega_{i,k-1})}^2
+ \sum_{\Omega_l \subset \Omega \setminus \Omega_{i,k-1}} \| \pi_l \psi_j^{(i)} \|_{s_l}^2 \right),
\end{align*}
where the constant $C_I$ depends on the continuity of the nodal interpolant $I^h(v)$ in $H^1$-seminorm
and $L^2$-norm and the assumption on $\mathcal{T}_h$ in \eqref{assume:Th} is used.
\end{lemma}
\begin{proof}
The estimates can be shown by using Cauchy-Schwarz inequality and by using the inequality
$$\| v \|_{s_l}^2 = \| (1-\pi_l ) v \|_{s_l}^2 + \| \pi_l v \|_{s_l}^2  \le \Lambda^{-1} |v|_{a(\Omega_l)}^2 +
\| \pi_l v \|_{s_l}^2.$$
The term $I^h((1-\chi_i^k)\psi_j^{(i)})$ is estimated for each triangle in $\mathcal{T}_h$ using
the assumption~\eqref{assume:Th}, the continuity of the nodal interpolant, and $\theta_l$ in $V_h$, i.e., $|\nabla \theta_l|$
are constant in each triangle.
\end{proof}

\begin{lemma}\label{lemma:kcbyk2}
For $\psi_j^{(i)}$ and $k\ge2$, we obtain that
\begin{align}
&\| \psi_j^{(i)} \|_{a(\Omega \setminus \Omega_{i,k})}^2
+\sum_{\Omega_l \subset \Omega \setminus \Omega_{i,k}} \| \pi_l \psi_j^{(i)} \|_{s_l}^2 \nonumber \\
&\le C (1+ \Lambda^{-1} ) \left( \| \psi_j^{(i)} \|_{a(\Omega_{i,k} \setminus \Omega_{i,k-2})}^2
+ \sum_{\Omega_l \subset \Omega_{i,k} \setminus \Omega_{i,k-2}} \| \pi_l \psi_j^{(i)} \|_{s_l}^2 \right). \label{bd:star3}
\end{align}
\end{lemma}
\begin{proof}
We recall the equation in \eqref{eqn:relaxed-energy-min} and choose
$v=I^h((1-\chi_{i}^{k-1}) \psi_j^{(i)})$ in $V_h$ to obtain
\begin{align*}
&a(\psi_j^{(i)}, I^h( (1-\chi_i^{k-1})\psi_j^{(i)}) ) + \sum_l s_l (\pi_l \psi_j^{(i)}, \pi_l (I^h((1-\chi_{i}^{k-1}) \psi_j^{(i)}) ) )\\
&=s_i(\phi_j^{(i)}, \pi_i (I^h((1-\chi_{i}^{k-1}) \psi_j^{(i)}) )).\end{align*}
In the above, the function $\phi_j^{(i)}$ is in $V_h(\Omega_i)$ and thus the right hand side vanishes
when $k-1 \ge 1$, i.e.,
\begin{equation}\label{identity:zero}
a(\psi_j^{(i)}, I^h( (1-\chi_i^{k-1})\psi_j^{(i)}) + \sum_l s_l (\pi_l \psi_j^{(i)}, \pi_l (I^h((1-\chi_{i}^{k-1}) \psi_j^{(i)}) ) )=0
\end{equation}

We now consider
$$a(\psi_j^{(i)}, I^h( (1-\chi_i^{k-1}) \psi_j^{(i)} ) )
=| \psi_j^{(i)} |_{a(\Omega \setminus \Omega_{i,k-1}^\delta)}^2 + \int_{\Omega_{i,k-1}^\delta \setminus
\Omega_{i,k-1}^{\delta-}} \rho \nabla \psi_{j}^{(i)} \cdot \nabla I^h ( (1-\chi_i^{k-1}) \psi_j^{(i)} ) \, dx.$$
For the second term above, we obtain that
\begin{align*}
&\int_{\Omega_{i,k-1}^\delta \setminus\Omega_{i,k-1}^{\delta-}} \rho \nabla \psi_{j}^{(i)} \cdot \nabla I^h ( (1-\chi_i^{k-1}) \psi_j^{(i)} ) \, dx \\
&\le | \psi_j^{(i)} |_{a(\Omega_{i,k-1}^\delta \setminus \Omega_{i,k-1}^{\delta-})}
C_I | (1-\chi_i^{k-1}) \psi_j^{(i)} |_{a(\Omega_{i,k-1}^\delta \setminus \Omega_{i,k-1}^{\delta-})}.
\end{align*}
Combining the two above, we finally obtain that
\begin{align}
&|\psi_j^{(i)} |_{a(\Omega \setminus \Omega_{i,k-1}^\delta)}^2 \nonumber \\
&=a(\psi_j^{(i)}, I^h( (1-\chi_i^{k-1}) \psi_j^{(i)} ) )-\int_{\Omega_{i,k-1}^\delta \setminus
\Omega_{i,k-1}^{\delta-}} \rho \nabla \psi_{j}^{(i)} \cdot \nabla I^h ( (1-\chi_i^{k-1}) \psi_j^{(i)} ) \, dx \nonumber \\
&\le a(\psi_j^{(i)}, I^h( (1-\chi_i^{k-1}) \psi_j^{(i)} ) ) + C_I | \psi_j^{(i)} |_{a(\Omega_{i,k-1}^\delta \setminus \Omega_{i,k-1}^{\delta-})}^2 + C_I \sum_{\Omega_l \subset \Omega_{i,k} \setminus \Omega_{i,k-2}} \| \psi_j^{(i)} \|_{s_l}^2. \label{bd:star1}
\end{align}

We consider
\begin{align*}
\sum_l s_l ( \pi_l \psi_j^{(i)}, \pi_l ( I^h ( (1-\chi_i^{k-1}) \psi_j^{(i)} ) ) )
&\quad =\sum_{\Omega_l \subset \Omega \setminus \Omega_{i,k} } s_l (\pi_l \psi_j^{(i)}, \pi_l \psi_j^{(i)} )\\
&+ \sum_{\Omega_l \subset \Omega_{i,k} \setminus \Omega_{i,k-2} } s_l (\pi_l \psi_j^{(i)}, \pi_l (I^h( (1-\chi_i^{k-1}) \psi_j^{(i)} )) ).
\end{align*}
Using the above identity, we obtain that
\begin{align}
& \sum_{\Omega_l \subset \Omega \setminus \Omega_{i,k} } \| \pi_l \psi_j^{(i)} \|_{s_l}^2 \nonumber \\
& = \sum_l s_l ( \pi_l \psi_j^{(i)}, \pi_l ( I^h ( (1-\chi_i^{k-1}) \psi_j^{(i)} ) ) )
- \sum_{\Omega_l \subset \Omega_{i,k} \setminus \Omega_{i,k-2} } s_l (\pi_l \psi_j^{(i)}, \pi_l (I^h( (1-\chi_i^{k-1}) \psi_j^{(i)} )) ) \nonumber \\
& \le \sum_l s_l ( \pi_l \psi_j^{(i)}, \pi_l ( I^h ( (1-\chi_i^{k-1}) \psi_j^{(i)} ) ) ) \nonumber \\
& + C_I \sum_{\Omega_l \subset \Omega_{i,k} \setminus \Omega_{i,k-2}} \| \pi_l \psi_j^{(i)} \|_{s_l}^2
+ C_I \Lambda^{-1} | \psi_j^{(i)} |_{a(\Omega_{i,k}\setminus \Omega_{i,k-2})}^2. \label{bd:star2}
\end{align}

From \eqref{bd:star1} and \eqref{bd:star2} with \eqref{identity:zero}, we finally obtain
\begin{align*}
& | \psi_j^{(i)} |_{a(\Omega \setminus \Omega_{i,k-1}^{\delta})}^2
+ \sum_{\Omega_l \subset \Omega \setminus \Omega_{i,k}} \| \pi_l \psi_j^{(i)} \|_{s_l}^2 \\
& \le C ( 1+\Lambda^{-1} ) \left( | \psi_j^{(i)} |_{a(\Omega_{i,k}\setminus \Omega_{i,k-2})}^2
+ \sum_{\Omega_l \subset \Omega_{i,k} \setminus \Omega_{i,k-2} } \| \pi_l \psi_j \|_{s_l}^2\right)
\end{align*}
and thus the resulting estimate.
\end{proof}

Using the estimate in Lemma~\ref{lemma:kcbyk2}, we can obtain the following
$$|\psi_j^{(i)}|_{a(\Omega \setminus \Omega_{i,k-2})}^2
+ \sum_{\Omega_l \subset \Omega \setminus \Omega_{i,k-2}} \| \pi_l \psi_j^{(i)} \|_{s_l}^2
\ge E \left( | \psi_j^{(i)} |_{a(\Omega \setminus \Omega_{i,k})}^2
+ \sum_{\Omega_l \subset \Omega \setminus \Omega_{i,k} } \| \pi_l \psi_j^{(i)} \|_{s_l}^2 \right),$$
where $E=(1+ (C(1+\Lambda^{-1}))^{-1} )>1$.
Applying the above estimate recursively, we obtain the following exponential decay property for $\psi_j^{(i)}$:
\begin{lemma}\label{lemma:estimate:decay}
For $k=2m$ with $m\ge1$, we obtain
$$ | \psi_j^{(i)} |_{a(\Omega \setminus \Omega_{i,k})}^2
+ \sum_{\Omega_l \subset \Omega \setminus \Omega_{i,k}} \| \pi_l \psi_j^{(i)} \|_{s_l}^2
\le E^{-m} \left( | \psi_j^{(i)} |_{a(\Omega)}^2 + \sum_l \| \pi_l \psi_j^{(i)} \|_{s_l}^2 \right),$$
where $E=(1+(C(1+\Lambda^{-1}))^{-1})>1$.
\end{lemma}

We recall the main estimate in \eqref{estimate:main} and let
$$w_i=\sum_{j=1}^{l_i} c_{ij} (\psi_j^{(i)}-\Psi_{j,ms}^{(i)}),\quad w(=u_0-u_{ms})=\sum_{i=1}^N w_i.$$
We then obtain the following result:
\begin{lemma}\label{global:to:local}
For $w=\sum_i w_i$ with $w_i=\sum_{j=1}^{l_i} c_{ij} ( \psi_j^{(i)}-\Psi_{j,ms}^{(i)} )$,
$$a(w,w)+\sum_l s_l (\pi_l w, \pi_l w)
\le C_I (1+\Lambda^{-1}) k^{r} \sum_i \left(a(w_i,w_i)+ \sum_l s_l (\pi_l w_i, \pi_l w_i) \right),$$
where $r$ denotes the dimension of the domain $\Omega$, i.e., $\Omega \subset \mathbb{R}^r$.
\end{lemma}
\begin{proof}
We consider
\begin{align}
& a(w,w)+\sum_l s_l (\pi_l w, \pi_l w) \nonumber \\
&= \sum_{i} \left( a(w_i,w) + \sum_l s_l (\pi_l w_i, \pi_l w) \right) \nonumber \\
&= \sum_{i} \left( a(w_i,I^h((1-\chi_i^{k+1}+\chi_{i}^{k+1}-\chi_{i}^{k}+\chi_{i}^{k})w))) \right. \nonumber \\
& \quad \quad \quad  \left. + \sum_l s_l ( \pi_l w_i, \pi_l I^h((1-\chi_i^{k+1}+\chi_i^{k+1}-\chi_{i}^{k}+\chi_{i}^{k})w)  ) \right)
\nonumber \\
&= \sum_{i} \left( a(w_i, I^h( (\chi_{i}^{k+1}-\chi_i^k)w))+\sum_l s_l (\pi_l w_i, \pi_l ( I^h( (\chi_i^{k+1}-\chi_i^k) w ) ) ) \right),
\label{eqn:wi}
\end{align}
where we have used that
$$a(w_i,I^h((1-\chi_i^{k+1})w))+\sum_l s_l(\pi_l w_i, \pi_l (I^h((1-\chi_i^{k+1})w)))=0$$
and
$$a(w_i,I^h( \chi_i^{k}w ))+\sum_l s_l(\pi_l w_i, \pi_l (I^h( \chi_i^{k}w)))=0.$$

For the first term in \eqref{eqn:wi}, we obtain that
\begin{align}
a(w_i,I^h( (\chi_i^{k+1}-\chi_i^k) w )) &\le |w_i|_{a(\Omega)} |I^h((\chi_i^{k+1}-\chi_i^k)w)|_{a(\Omega)} \nonumber\\
&\le |w_i|_{a(\Omega)} C_I^{1/2} |(\chi_i^{k+1}-\chi_i^k)w|_{a(\Omega)}. \nonumber
\end{align}
For the second term in \eqref{eqn:wi}, we obtain that
\begin{align}
& \sum_l s_l (\pi_l w_i, \pi_l ( I^h( (\chi_i^{k+1}-\chi_i^k) w ) ) ) \nonumber \\
& \le \left( \sum_l \| \pi_l w_i \|_{s_l}^2 \right)^{1/2} \left( \sum_l \| \pi_l( I^h((\chi_i^{k+1}-\chi_i^k)w) ) \|_{s_l}^2  \right)^{1/2} \nonumber \\
& \le \left( \sum_l \| \pi_l w_i \|_{s_l}^2 \right)^{1/2} \left( C_I \sum_l \|  (\chi_i^{k+1}-\chi_i^k)w  \|_{s_l}^2 \right)^{1/2}. \nonumber
\end{align}
Combining \eqref{eqn:wi} with the two estimates above and using Cauchy-Schwarz inequality,
\begin{align*}
& a(w,w)+\sum_l s_l (\pi_l w, \pi_l w) \\
& \le C_I^{1/2} \left( \sum_i \left( |w_i|_{a(\Omega)}^2 + \sum_l \| \pi_l w_i \|_{s_l}^2 \right) \right)^{1/2} \\
& \quad \quad \quad \quad \left( \sum_i \left( | (\chi_i^{k+1}-\chi_i^k)w |_{a(\Omega)}^2 + \sum_l \| (\chi_i^{k+1}-\chi_i^k) w \|_{s_l}^2 \right) \right)^{1/2} \\
& \le C_I^{1/2} \left( \sum_i \left( |w_i|_{a(\Omega)}^2 + \sum_l \| \pi_l w_i \|_{s_l}^2 \right) \right)^{1/2} \\
& \quad \quad \quad \quad \left( \sum_i \left( (1+\Lambda^{-1}) | w |_{a(\Omega_{i,k+2}\setminus\Omega_{i,k-1})}^2  + \sum_{\Omega_l \in \Omega_{i,k+2}\setminus \Omega_{i,k-1}} \| \pi_l w  \|_{s_l}^2 \right) \right)^{1/2}\\
& \le C_I^{1/2} k^{r/2} (1+\Lambda^{-1})^{1/2} \left( \sum_i \left( |w_i|_{a(\Omega)}^2 + \sum_l \| \pi_l w_i \|_{s_l}^2 \right) \right)^{1/2} \left( |w|_{a(\Omega)}^2 +\sum_l \| \pi_l w \|_{s_l}^2  \right)^{1/2},
\end{align*}
where $r$ denotes the dimension of the model problem, i.e., $\Omega \subset \mathbb{R}^r$.
\end{proof}

We note that the estimate in Lemma~\ref{lemma:estimate:difference} holds for $w_i$ and using the estimate combined with
Lemma~\ref{lemma:estimate:decay},
$$ a(w_i,w_i) + \sum_l s_l (\pi_l w_i, \pi_l w_i)
\le C_I (1+\Lambda^{-1}) E^{-m} \left( a(u_{0,i},u_{0,i}) +  \sum_l s_l (\pi_l u_{0,i}, \pi_l u_{0,i} ) \right),$$
where $u_{0,i}=\sum_{j=1}^{l_i} c_{ij} \psi_j^{(i)}$ and $k=2m+1$.
Combining this with Lemma~\ref{global:to:local}, we obtain the following key estimate:
\begin{lemma}\label{lemma:key:estimate}
For $w=\sum_i \sum_{j=1}^{l_i} c_{ij} ( \psi_j^{(i)}-\psi_{j,ms}^{(i)} )$ and $k=2m+1$ with $m \ge 1$,
we obtain that
$$a(w,w)+\sum_l s_l (\pi_l w, \pi_l w)
\le C_I^2 (1+\Lambda^{-1})^2 k^{r} E^{-m} \sum_i \left(  a(u_{0,i},u_{0,i}) +  \sum_l s_l (\pi_l u_{0,i}, \pi_l u_{0,i} ) \right),$$
where $E$ is bigger than 1 and depends on $\Lambda$.
\end{lemma}

We will now work on obtaining the following estimate
$$\sum_{i=1}^N \left( a(u_{0,i},u_{0,i}) +  \sum_l s_l (\pi_l u_{0,i}, \pi_l u_{0,i} )  \right)
\le D \left( |u_0|_{a(\Omega)}^2 + \sum_l \| \pi_l u_0 \|_{s_l}^2 \right),$$
where $u_0=\sum_i u_{0,i}$ and $D$ denotes a constant related to the function $z$ introduced below after \eqref{eqn:u0}.

We let $$\Phi_i:=\sum_{j=1}^{l_i} c_{ij} \phi_j^{(i)},$$
and we can then observe that
\begin{equation}\label{eqn:u0}
a(u_0,v)+\sum_l s_l(\pi_l u_0, \pi_l v)=\sum_{i} s_i(\Phi_i, \pi_i v),\quad \forall v \in V_h.
\end{equation}
We can choose $z$ in $V_h$ (Lemma 2 in \cite{CEL2018})
such that
$$\pi_i z = \Phi_i,\quad \forall i=1,\cdots,N,\quad |z|_{a(\Omega)}^2 + \sum_i \|\pi_i z \|_{s_i}^2
\le D \sum_i \| \Phi_i \|_{s_i}^2.$$
Choosing $v=z$ in \eqref{eqn:u0} and using Cauchy-Schwarz inequality, we obtain that
$$\sum_i \| \Phi_i \|_{s_i}^2 \le D \left( |u_0|_{a(\Omega)}^2 + \sum_l \| \pi_l u_0 \|_{s_l}^2 \right).$$

We introduce the constant $C_p$ satisfying
$$\frac{ \max_{l} ( \max_{\tau \in \Omega_l \setminus \Omega_l^{INT}}
( \min_{x \in \tau}( \rho(x)  \sum_{j \in n(l)} |\nabla \theta_j (x)|^2 ) ) ) } { \min_{\tau \in \mathcal{T}_h} (\max_{x \in \tau} \rho(x)) } \le C_p,$$
where $\Omega_l^{INT}$ denotes the union of triangles only belonging to $\Omega_l$, i.e., $\theta_l(x)=1$ for all $x$
in $\Omega_l^{INT}$.
We then finally obtain
\begin{equation}\label{bd:Phii}
\sum_i \| \Phi_i \|_{s_i}^2 \le D (1+C_p) | u_0 |_{a(\Omega)}^2,
\end{equation}
where we have used the Poincar\'{e} inequality
$$\int_{\Omega} u_0^2  \, dx \le C \int_{\Omega} |\nabla u_0|^2 \, dx,$$
the assumption~\eqref{assume:Th} for the triangulation $\mathcal{T}_h$,
and the property for the partition of unity functions, i.e.,  $\sum_{j \in n(l)} |\nabla \theta_j(x) |^2$
is constant in each $\tau$ in $\mathcal{T}_h$. We note that $n(l)$ denotes the set of overlapping subdomain indices $j$,
such that $\Omega_l$ intersects with $\Omega_j^\prime$.

We now observe that
$$
a(u_{0,i},u_{0,i}) +  \sum_l s_l (\pi_l u_{0,i}, \pi_l u_{0,i} ) = s_i(\Phi_i, \pi_i u_{0,i}).$$
From the above identity and then applying Cauchy-Schwarz inequality, we obtain
$$a(u_{0,i},u_{0,i}) + \sum_l s_l (\pi_l u_{0,i}, \pi_l u_{0,i}) \le s_i(\Phi_i,\Phi_i).$$
Taking summation on $i$ and using the estimate in \eqref{bd:Phii},
\begin{align}
& \sum_i \left(  | u_{0,i} |_{a(\Omega)}^2 + \sum_l \| \pi_l u_{0,i} \|_{s_l}^2 \right) \nonumber \\
& \le \sum_i \| \Phi_i \|_{s_i}^2 \nonumber \\
& \le D (1+C_p) | u_0 |_{a(\Omega)}^2. \label{bd:sum:u0i}
\end{align}
Combining Lemma~\ref{lemma:key:estimate} with \eqref{bd:sum:u0i}, we finally obtain:
\begin{lemma}\label{lemma:ket:etimate:result}
For $w=\sum_i \sum_{j=1}^{l_i} c_{ij} ( \psi_j^{(i)}-\psi_{j,ms}^{(i)} )$ and $k=2m+1$ with $m \ge 1$,
we obtain
$$a(w,w)+\sum_l s_l (\pi_l w, \pi_l w)
\le C_I^2 (1+\Lambda^{-1})^2 k^{r} E^{-m} D (1+C_p) a(u_0,u_0),$$
where $E$ is bigger than 1 and depends on $\Lambda$.
\end{lemma}

From the above Lemma we can obtain the main estimate in \eqref{estimate:main} with the constant $C_0^2$
to be
$$C_0^2=C_I^2 (1+\Lambda^{-1})^2 k^{r} E^{-m} D (1+C_p).$$
We note that by choosing sufficiently large $k=2m+1$, one can control the constant $C_p$
to obtain the resulting bound on $C_0$ robust to both $\rho(x)$ and $\delta$, the overlapping
width in the partition.

\begin{theorem}\label{thm:stable:decomposition}
For a sufficiently large $k=2m+1$, depending on $\rho(x)$ and the partition of unity functions, $\theta_l(x)$,
one can obtain the following bound for $u_{ms}$ in \eqref{def:ums} and $u_i$ in \eqref{def:ui},
$$\sum_i a(u_i,u_i) + a(u_{ms},u_{ms}) \le C_0^2 a(u,u),$$
where the constant $C_0$ depends on $\Lambda$ but does not depend on $\rho(x)$ and the overlapping width $\delta$ in the subdomain partition.
\end{theorem}

\section{Numerical experiment}\label{sec:numerics}

In our numerical experiments, we form the nonoverlapping subdomain partitions and overlapping subdomain partitions
as follows.
We partition the unit square domain into $n\times n$ uniform squares
to obtain a nonoverlapping subdomain partition.
We use the notation $\Omega_i$ for each of them.
Each nonoverlapping subdomain is divided into uniform triangles by using $m \times m$ uniform squares
and then dividing each square into two triangles.
We define $\Omega_{i,d,h}$ (see Fig. \ref{fig1:subfig:a}) by enlarging $\Omega_i$ by $d$ fine grid layers and let $\Omega_i^\prime:=\Omega_{i,d,h}$. We note that $\{ \Omega_i^\prime \}_i$ is an overlapping subdomain partition of $\Omega$, and will be used when forming local problems in the preconditioner.
In addition, we define $\Omega_{i,k,H}^\delta$ (see Fig. \ref{fig1:subfig:b}) by enlarging $\Omega_i$ by $k$ layers of neighboring subdomains and then by $d$ fine layers, i.e. $\delta=dh$.
We let $\widetilde{\Omega}_i:=\Omega_{i,k,H}^\delta$.
We note that the relaxed constrained minimization problem for finding $\psi_{j,ms}^{(i)}$
is solved in the smaller region in $\widetilde{\Omega}_i$ rather than in the whole domain $\Omega$,
which greatly helps to reduce the computational cost.
We remark that the basis functions $\psi_{j,ms}^{(i)}$ are used in the global coarse problem of our preconditioner.

\begin{figure}[t]
\centering
\subfigure[The red domain is $\Omega_{i,d,h}$ with $d=1$.]{ \label{fig1:subfig:a} %% label for first subfigure
\includegraphics[width=2.2in]{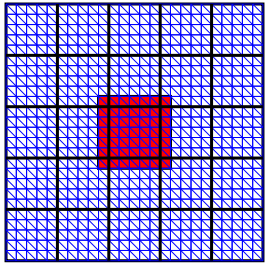}}
\hspace{1in}
\subfigure[The green domain is $\Omega_{i,k,H}^\delta$ with $k=1$ and $\delta=h$.]{ \label{fig1:subfig:b} %% label for second subfigure
\includegraphics[width=2.2in]{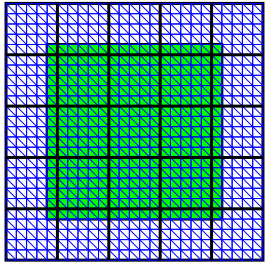}}
{\caption{Examples of an overlapping subdomain $\Omega_i^\prime:=\Omega_{i,d,h}$ and a subregion
$\widetilde{\Omega}_i:=\Omega_{i,k,H}^\delta$ containing $\Omega_{i}^\prime$.}
\label{Fig:submesh-1}} %% label for entire figure
\end{figure}

We first consider the coarse problem formed by directly using the functions
in the auxiliary space $V_{aux}$. In this case, we find $\phi_j^{(i)}$ in $V_{aux}$
by solving the following eigenvalue problem in each overlapping subdomain $\Omega_i^\prime=\Omega_{i,d,h}$, i.e.,
$$\int_{\Omega_i^\prime} \rho(x) \nabla v \nabla w \, dx = \lambda \int_{\Omega_i^\prime} \rho(x) | \nabla \theta_i(x)|^2
v w \, dx,\quad \forall w \in V(\Omega_i^\prime)$$
and we introduce the coarse basis
functions given by $I^h(\theta_i \phi_j^{(i)})$. This approach is
similar to that considered in \cite{Galvis1}.
%We note that in our method we solve the above eigenvalue problem
%in each nonoverlapping subdomain.
In Table~\ref{TB1:oldmethod:uniform},
we present the performance of the two-level overlapping Schwarz method
for a model problem with a uniform constant $\rho(x)=1$ and
in Table~\ref{TB2:oldmethod:random} for the case with $\rho(x)$ being
highly random in the range $(1,\; 10^6)$.
We observe that the minimum eigenvalues seem to robust to the contrast
in the coefficient $\rho(x)$ while they clearly show the dependence
on the overlapping width, $2 \delta$ with $\delta=dh$ for both examples
in Tables~\ref{TB1:oldmethod:uniform} and \ref{TB2:oldmethod:random}.

\begin{table}
\footnotesize \caption{Performance of coarse space from $V_{aux}$ for a model with  $\rho(x)=1$,
$\Lambda=1+\log (m+2d)$: iter (number of iterations),
$\lambda_{\min}$ (minimum eigenvalues), $\lambda_{\max}$ (maximum eigenvalues), $\kappa$ (condition numbers),
$pD$ (average number of coarse basis per subdomain).}
\label{TB1:oldmethod:uniform} \centering

\begin{tabular}{|c|c|c|c|c|c|c|}
  \hline
  $n(m)$ & $d$  &  $iter$ & $\lambda_{min}$  &  $\lambda_{max} $  & $\kappa$ & $pD$\\
 \hline
  6(10) & 1  &   42 &    0.13         &  4.00  &30.51 &1     \\
        & 2  &   38 &    0.18         &  4.17  &23.80 &1   \\
        & 3  &   34 &    0.22         &  4.35  &19.70 &1   \\
        & 4  &   30 &    0.27         &  4.53  &16.58 &1   \\
        & 5  &   28 &    0.34         &  4.71  &14.01 &1   \\
 % \hline
%  8(10) & 1  &   46 &    0.13        &   4.00  &31.98 &1 \\
%        & 2  &   43 &    0.16        &   4.17  &26.69 &1  \\
%        & 3  &   38 &    0.19        &   4.36  &23.36 &1  \\
%        & 4  &   35 &    0.22        &   4.55  &20.58 &1   \\
%        & 5  &   32 &    0.26        &   4.72  &18.10 &1   \\
  \hline
  10(10) & 1  &  47  &   0.12         &  4.00  &32.66&1     \\
         & 2  &  45  &   0.15         &  4.17  &28.28&1  \\
         & 3  &  42  &   0.17         &  4.37  &25.60&1  \\
         & 4  &  38  &   0.20         &  4.55  &23.27&1   \\
         & 5  &  35  &   0.22         &  4.72  &21.11&1   \\
  \hline

\end{tabular}
\end{table}

\begin{table}
\footnotesize \caption{Performance of coarse space from $V_{aux}$ for a model with  highly random $\rho(x)$ in $(1,\,10^6)$,
$\Lambda=1+\log (m+2d)$: iter (number of iterations),
$\lambda_{\min}$ (minimum eigenvalues), $\lambda_{\max}$ (maximum eigenvalues), $\kappa$ (condition numbers),
$pD$ (average number of coarse basis per subdomain).}
\label{TB2:oldmethod:random} \centering
\begin{tabular}{|c|c|c|c|c|c|c|}
  \hline
  $n(m)$ & $d$  &  $iter$ & $\lambda_{min}$  &  $\lambda_{max} $  & $\kappa$ & $pD$\\
 \hline
  6(10) & 1  &   64 &    0.06         &  5.00  &  85.05  &6.17     \\
        & 2  &   47 &    0.14         &  5.00  &  36.66  &6.75   \\
        & 3  &   38 &    0.30         &  5.00  &  16.84  &6.69   \\
        & 4  &   32 &    0.51         &  4.99  &   9.71  &7.75   \\
        & 5  &   30 &    0.47         &  4.99  &  10.52  &7.28   \\
 % \hline
%  8(10) & 1  &   65 &    0.11        &   5.00  &  46.92  &6.52 \\
%        & 2  &   49 &    0.17        &   5.00  &  29.40  &6.92  \\
%        & 3  &   42 &    0.24        &   4.99  &  21.19  &7.50  \\
%        & 4  &   35 &    0.42        &   4.99  &  11.83  &7.31   \\
%        & 5  &   32 &    0.50        &   4.99  &  10.07  &7.86   \\
  \hline
  10(10) & 1  &  77  &   0.07         &  5.00  &  75.10 &6.54     \\
         & 2  &  53  &   0.16         &  5.00  &  31.29 &7.31  \\
         & 3  &  37  &   0.35         &  4.99  &  14.36 &7.26  \\
         & 4  &  35  &   0.44         &  4.99  &  11.30 &7.78   \\
         & 5  &  32  &   0.50         &  4.99  &   9.99 &7.99   \\
  \hline

\end{tabular}

\end{table}

In Table~\ref{TB3:uniform:k13}, we present
the performance of the new coarse space $V_{ms}$ for the case with
$k=1$ and $k=3$, respectively, for the model with $\rho(x)=1$.
We can observe that the minimum eigenvalues are bigger than
those in Table~\ref{TB1:oldmethod:uniform}.
For the case with smaller
oversampling region with $k=1$, i.e., with only one layer of
neighboring subdomains, we can get quite robust minimum eigenvalues
even for the case with smaller overlaps.

\begin{table}
\footnotesize \caption{Performance of coarse space from $V_{ms}$ for a model with  $\rho(x)=1$,
$\Lambda=1+\log m$ and $k=1$ (upper table), $k=3$ (lower table): iter (number of iterations),
$\lambda_{\min}$ (minimum eigenvalues), $\lambda_{\max}$ (maximum eigenvalues), $\kappa$ (condition numbers),
$pD$ (average number of coarse basis per subdomain).}
\label{TB3:uniform:k13} \centering
\begin{tabular}{|c||c|c|c|c|c|c|c|c|}
  \hline
 k&  $n(m)$ & $d$ &  $iter$ & $\lambda_{min}$  &  $\lambda_{max} $  & $\kappa$ & $pD$\\
 \hline
 \multirow{10}{*}{1}    &   6(10) & 1  &   26 &    0.43         &  4.02  &9.29  &2.89     \\
  &      & 2  &   25 &    0.57         &  4.06  &7.16  &1   \\
  &     & 3  &   21 &    0.79         &  4.05  &5.12  &1   \\
  &      & 4  &   21 &    0.94         &  4.12  &4.36  &1   \\
  &      & 5  &   20 &    1.00         &  4.27  &4.27  &1   \\
  %\hline
% & 8(10) & 1  &   28 &    0.40        &   4.02  &10.02 &2.94 \\
% &       & 2  &   26 &    0.54        &   4.11  &7.56  &1  \\
% &       & 3  &   22 &    0.77        &   4.06  &5.27  &1  \\
% &       & 4  &   21 &    0.94        &   4.12  &4.38  &1   \\
% &       & 5  &   21 &    1.00        &   4.26  &4.27  &1   \\
& 10(10) & 1  &  27  &   0.39         &  4.03  &10.25&2.96     \\
  &       & 2  &  26  &   0.53         &  4.16  &7.81 &1  \\
  &       & 3  &  22  &   0.76         &  4.05  &5.33 &1  \\
  &       & 4  &  20  &   0.94         &  4.12  &4.38 &1   \\
  &       & 5  &  21  &   1.00         &  4.25  &4.26 &1   \\
  \hline
\multirow{10}{*}{3} &   6(10) & 1  &   25 &    0.47         &  4.01  &8.46  &2.89     \\
       & & 2  &   23 &    0.60         &  4.02  &6.65  &1   \\
       & & 3  &   21 &    0.80         &  4.06  &5.05  &1   \\
       & & 4  &   21 &    0.94         &  4.12  &4.36  &1   \\
       & & 5  &   20 &    1.00         &  4.27  &4.27  &1   \\
  %\hline
% & 8(10) & 1  &   26 &    0.43        &   4.01  &9.43  &2.94 \\
% &       & 2  &   24 &    0.57        &   4.02  &7.07  &1  \\
% &       & 3  &   22 &    0.78        &   4.06  &5.21  &1  \\
% &       & 4  &   21 &    0.94        &   4.12  &4.38  &1   \\
% &       & 5  &   21 &    1.00        &   4.26  &4.26  &1   \\
  %\hline
 & 10(10) & 1  &  27  &   0.41         &  4.01  &9.91 &2.96     \\
  &       & 2  &  24  &   0.55         &  4.02  &7.30 &1  \\
  &       & 3  &  22  &   0.77         &  4.06  &5.29 &1  \\
  &       & 4  &  20  &   0.94         &  4.11  &4.38 &1   \\
  &       & 5  &  21  &   1.00         &  4.26  &4.26 &1   \\
  \hline
\end{tabular}
\end{table}

In Table~\ref{TB5:random1:k13}, we present
the performance of the new coarse space $V_{ms}$ for the case with
$k=1$ and $k=3$, respectively, for the model with $\rho(x)$ in the range $(1,\; 10^6)$.
We can observe that our method gives more robust results for this example than those from the uniform case
considered in Table~\ref{TB3:uniform:k13}.
The minimum eigenvalues are less dependent on the overlapping width and they are quite close
to the value $1$ even when $d=2$.
In addition, the number of coarse basis per subdomain is also
increasing about one even for the smaller $d=1,2$, compared to $d=5$.
The condition numbers
seem more robust when the oversampling subregion size is larger, i.e., when $k=3$ for $10 \times 10$
subdomain partition.
For this case, the minimum eigenvalues are very close to one when $d=2,3,4,5$.
With the smaller $k=1$, we can still achieve quite good and robust results.
\begin{table}
\footnotesize \caption{Performance of coarse space from $V_{ms}$ for a model with random $\rho(x)$
in the range $(1,\, 10^6)$,
$\Lambda=1+\log m$ and $k=1$ (upper table), $k=3$ (lower table): iter (number of iterations),
$\lambda_{\min}$ (minimum eigenvalues), $\lambda_{\max}$ (maximum eigenvalues), $\kappa$ (condition numbers),
$pD$ (average number of coarse basis per subdomain).}
\label{TB5:random1:k13} \centering
\begin{tabular}{|c||c|c|c|c|c|c|c|}
  \hline
 $k$ &  $n(m)$ & $d$ &  $iter$ & $\lambda_{min}$  &  $\lambda_{max} $  & $\kappa$ & $pD$\\
 \hline
  \multirow{10}{*}{1} &6(10) & 1  &   22 &    0.84         &  4.23  &5.02  &5.83     \\
    &    & 2  &   25 &    0.87         &  4.99  &5.71  &5.75   \\
    &    & 3  &   24 &    0.98         &  4.99  &5.07  &5.50   \\
    &    & 4  &   24 &    1.00         &  4.99  &4.99  &4.97   \\
    &    & 5  &   24 &    1.00         &  4.98  &4.98  &4.42   \\
  %\hline
 % & 8(10) & 1  &   25 &    0.60        &   4.38  &7.32  &6.58 \\
%   &     & 2  &   24 &    0.95        &   4.98  &5.27  &6.44  \\
%   &     & 3  &   24 &    0.95        &   5.00  &5.25  &6.00  \\
%   &     & 4  &   24 &    0.97        &   4.99  &5.13  &5.50   \\
%   &     & 5  &   25 &    1.00        &   4.99  &5.02  &5.11   \\
  %\hline
  &10(10) & 1 &  25  &   0.65         &  4.15  &6.39  &6.35     \\
  &       & 2 &  26  &   0.83         &  4.99  &6.02  &6.14  \\
  &       & 3 &  25  &   0.87         &  4.99  &5.76  &5.60  \\
  &       & 4 &  25  &   0.88         &  4.99  &5.69  &5.19   \\
  &       & 5 &  25  &   0.89         &  4.99  &5.60  &4.71   \\
  \hline
\multirow{10}{*}{3} & 6(10) & 1  &   21 &    0.85   &  4.21  &4.96  &5.83     \\
                    & & 2  &   24 &    0.87         &  4.99  &5.71  &5.75   \\
                    & & 3  &   24 &    0.99         &  4.99  &5.03  &5.50   \\
                    & & 4  &   24 &    1.00         &  4.99  &4.99  &4.97   \\
                    & & 5  &   24 &    1.00         &  4.98  &4.98  &4.42   \\
 % \hline
%&  8(10) & 1 &   24 &    0.61        &   4.38  &7.16  &6.58 \\
%&        & 2 &   24 &    0.95        &   4.97  &5.22  &6.44  \\
%&        & 3 &   23 &    1.00        &   5.00  &5.00  &6.00  \\
%&        & 4 &   24 &    1.00        &   4.99  &4.99  &5.50   \\
%&        & 5 &   24 &    1.00        &   4.99  &4.98  &5.11   \\
%%  \hline
   & 10(10) & 1  &  22  &   0.73         &  4.15  &5.72  &6.35     \\
   &      & 2  &  24  &   0.92         &  4.99  &5.45  &6.14  \\
   &      & 3  &  23  &   0.96         &  4.99  &5.21  &5.60  \\
   &      & 4  &  24  &   1.00         &  4.99  &4.98  &5.19   \\
   &      & 5  &  24  &   1.00         &  4.99  &4.98  &4.71   \\
  \hline
\end{tabular}
\end{table}

In Table~\ref{TB5:random2:k13}, we present
the performance of the new coarse space $V_{ms}$ for the case with
$k=1$ and $k=3$, respectively, for the model with $\rho(x)$ in the range $(10^{-3},\; 10^3)$.
We can observe similar results to those in previous Table~\ref{TB5:random1:k13}.
Even with smaller oversampling subregion and with smaller overlapping width,
we can get quite good performance.
For example, with $k=1$ and $d=1$,
the proposed method require only one more iteration than the case with $k=3$ and $d=5$
when $10 \times 10$ subdomain partition is considered.
For this specific example, we observe that the number of coarse basis per subdomain
is 6.51 for the case with $k=1$ and $d=1$, while that is about 4.77 per subdomain for the case with $k=3$ and $d=5$.
With only one or two more coarse basis functions per subdomain, we can obtain a robust and efficient coarse problem.
We note that in our analysis we have shown that the oversampling size $k$ can be chosen large enough
to control the contrast in the coefficient and the gradient of partition of unity functions, otherwise
they will affect the resulting condition numbers.
We can conclude that in practice our method gives good performance even for the smaller oversampling
size, i.e, with $k=1$, and thus
the proposed method seems very robust to the contrast in the coefficient and the overlapping
width in the subdomain partition.

\begin{table}
\footnotesize \caption{Performance of coarse space from $V_{ms}$ for a model with random $\rho(x)$
in the range $(10^{-3},\, 10^3)$,
$\Lambda=1+\log m$ and $k=1$ (upper table), $k=3$ (lower table): iter (number of iterations),
$\lambda_{\min}$ (minimum eigenvalues), $\lambda_{\max}$ (maximum eigenvalues), $\kappa$ (condition numbers),
$pD$ (average number of coarse basis per subdomain).}
\label{TB5:random2:k13} \centering

\begin{tabular}{|c||c|c|c|c|c|c|c|}
  \hline
 $k$ &  $n(m)$ & $d$ & $iter$ & $\lambda_{min}$  &  $\lambda_{max} $  & $\kappa$ & $pD$\\
 \hline
 \multirow{10}{*}{1} &  6(10) & 1  &   23 &    0.69         &  4.23  &6.13  &5.83     \\
      &  & 2  &   24 &    0.91         &  4.98  &5.49  &5.67   \\
      &  & 3  &   24 &    0.97         &  5.00  &5.13  &5.69   \\
      &  & 4  &   24 &    1.00         &  5.00  &5.00  &5.14   \\
      &  & 5  &   24 &    1.00         &  4.98  &4.98  &4.67   \\
 % % \hline
% & 8(10) & 1  &   23 &    0.74        &   4.25  &5.74  &6.81 \\
% &       & 2  &   24 &    0.94        &   4.99  &5.29  &6.47  \\
% &       & 3  &   24 &    0.96        &   4.99  &5.18  &6.19  \\
% &       & 4  &   25 &    0.99        &   5.00  &5.04  &5.45   \\
% &       & 5  &   25 &    0.96        &   4.99  &5.20  &4.77   \\
% % \hline
 & 10(10) & 1  &  26  &   0.52         &  4.45  &8.57 &6.51     \\
 &        & 2  &  25  &   0.77         &  4.99  &6.51 &6.42  \\
 &        & 3  &  26  &   0.82         &  4.98  &6.07 &6.13  \\
 &        & 4  &  25  &   0.91         &  4.99  &5.46 &5.33   \\
 &        & 5  &  25  &   0.96         &  4.99  &5.20 &4.77   \\
  \hline
 \multirow{10}{*}{3} &  6(10) & 1  &   23 &    0.69         &  4.23  &6.12  &5.83     \\
     &   & 2  &   24 &    0.91         &  4.98  &5.48  &5.67   \\
     &   & 3  &   24 &    1.00         &  5.00  &5.01  &5.69   \\
     &   & 4  &   24 &    1.00         &  5.00  &5.00  &5.14   \\
     &   & 5  &   24 &    1.00         &  4.98  &4.98  &4.67   \\
 % \hline
% & 8(10) & 1  &   22 &    0.73        &   4.25  &5.82  &6.81 \\
% &       & 2  &   24 &    0.97        &   4.99  &5.15  &6.47  \\
% &       & 3  &   24 &    0.96        &   4.99  &5.20  &6.19  \\
% &       & 4  &   24 &    1.00        &   4.99  &4.99  &5.45   \\
% &       & 5  &   24 &    1.00        &   4.99  &4.99  &4.77   \\
%  \hline
 & 10(10) & 1  &  25  &   0.52         &  4.45  &8.55 &6.51     \\
 &        & 2  &  25  &   0.88         &  4.98  &5.65 &6.42  \\
 &        & 3  &  24  &   1.00         &  4.98  &4.98 &6.13  \\
 &        & 4  &  24  &   1.00         &  4.98  &4.98 &5.33   \\
 &        & 5  &  25  &   1.00         &  4.99  &4.99 &4.77   \\
  \hline
\end{tabular}

\end{table}

%\bibliography{os-gmfem}
%\bibliographystyle{plain}

\end{document}